\documentclass[]{amsart}

\usepackage{epstopdf}
\usepackage[caption=false]{subfig}

\usepackage[numbers,sort&compress]{natbib}
\bibpunct[, ]{[}{]}{,}{n}{,}{,}
\makeatletter
\def\NAT@def@citea{\def\@citea{\NAT@separator}}
\makeatother

\theoremstyle{plain}
\newtheorem{theorem}{Theorem}[section]
\newtheorem{lemma}[theorem]{Lemma}
\newtheorem{corollary}[theorem]{Corollary}
\newtheorem{proposition}[theorem]{Proposition}

\theoremstyle{definition}

\theoremstyle{remark}
\newtheorem{remark}{Remark}

\newcommand{\Hq}{\mathbb H}
\newcommand{\Sq}{\mathbb S}
\newcommand{\C}{\mathbb C}
\newcommand{\R}{\mathbb R}

\newcommand{\norm}[1]{\left\Vert#1\right\Vert}
\newcommand{\scal}[1]{\left<#1\right>}
    \newcommand{\fin}{\hfill $\square$}
    \newcommand{\bz}{\overline{z}}
\newcommand{\bq}{\overline{q}}

\begin{document}
%
%

\title{On some analytic properties of slice poly-regular Hermite polynomials}

\author{Amal El Hamyani and  Allal Ghanmi}
\address{CeReMAR, A.G.S., L.A.M.A., Department of Mathematics, P.O. Box 1014, 
           Faculty of Sciences, Mohammed V University in Rabat, Morocco}
\email{ag@fsr.ac.ma}

\maketitle

\begin{abstract}
We consider a quaternionic analogue of the univariate complex Hermite polynomials and study some of their analytic properties in some detail. We obtain their integral representation as well as the operational formulas of exponential and Burchnall types they obey. We show that they form an orthogonal basis of both the slice and the full Hilbert spaces on the quaternions with respect to the Gaussian measure. We also provide different types of generating functions. Remarkably identities, including quadratic recurrence formulas of Nielsen type are derived.
\end{abstract}

\section{Introduction}
The univariate complex Hermite polynomials for a pair of conjugate variables,
         \begin{align}\label{chp}
         H_{m,n}(z,\bz )=(-1)^{m+n}e^{|z|^2 }\dfrac{\partial ^{m+n}}{\partial \bz^{m} \partial z^{n}} e^{-|z|^2} ,
         \end{align}
 constitute a complete orthogonal system of the Hilbert space $L^2(\C;e^{-|z|^2}dxdy)$.
 Such polynomials have found several interesting applications in various branches of mathematics,
 physics and technology. They are used as basic tools in studying the complex Markov process \cite{Ito52},
 the nonlinear analysis of travelling wave tube amplifiers \cite{Barrett}, the singular values of the Cauchy transform \cite{IntInt06},
 the coherent states \cite{AliBagarelloHonnouvo10,AliBagarelloGazeau13,ThiruloAli13}, combinatory
 \cite{IsmailSimeonovProc2015,IsmailTrans2016}, the poly-analytic functions and signal processing \cite{RaichZhou04,DallingerRuotsalainenWichmanRupp10}.

A natural extension of $H_{m,n}(z,\bz )$ to the quaternionic setting is defined by \cite{ThiruloAli13}
\begin{align} \label{explicitIntro}
H^Q_{m,n}(q,\bq ) & = \sum_{k=0}^{m\wedge n} (-1)^k k! \binom{m}{k}  \binom{n}{k}   q^{m-k} \bq^{n-k}, \quad q\in\Hq.
\end{align}
This class appears naturally when investigating spectral properties of the sliced second order differential operator of Laplacian type
$ 
\Delta_{q}=-\partial_s \overline{\partial}_s + \bq \overline{\partial}_s,
$ 
where $\partial_{s}$ and $\overline{\partial}_{s}$ are the left slice derivatives in \eqref{cullend}
and \eqref{cullendb}, respectively.
More precisely,
the polynomials $H^Q_{m,n}$, for varying $n$, form an orthogonal basis of a new class of slice poly-hyperholomorphic Bargmann-Fock space
 $\mathcal{GB}_{m}^{2}(\Hq)$. The particular case of $m=0$ leads to the standard slice hyperholomorphic Bargmann-Fock space introduced
 in \cite{AlpayColomboSabadiniSalomon2014} which is the quaternionic analogue of the holomorphic Bargmann-Fock space.
 Moreover, such polynomials are employed to determinate the explicit formula for the corresponding reproducing kernel
 and in introducing the corresponding Segal-Bargmann transforms \cite{elhamyani2017}.

In the present paper, we will study them in some detail and obtain further remarkably interesting properties.
Mainely, we investigate the following items:
 \begin{itemize}
 \item  Integral representation,
 \item  Alternatives to Rodrigues' formula,
 \item  Exponential and Burchnall operational formulas,
 \item  Orthogonality property,
 \item  Generating  functions,
 \item  Quadratic recurrence identities, 
 \end{itemize}
 among others.
Thus, different realizations are given and operational formulas of Burchnall type they satisfy are obtained.
These formulas are then employed to obtain quadratic recurrence formulas and generating functions. The orthogonality in the full Hilbert space on $\Hq$ is considered. Moreover, we show that these polynomials form an orthogonal basis.

Although the extension to quaternion variable is natural, some of the obtained results 
are quite diverse and require special attention, mainly due to the peculiarities of the quaternionic setting.
In fact, we must take care of the non-commutativity of the product in $\Hq$ and the notion of the left slice derivative.
Both make the study of $H^Q_{m,n}(q,\bq)$ distinctively different from the one provided for the complex Hermite polynomials.
 This will be clarified when dealing with the associated generating functions showing in what the quaternionic structure is specially important.
We provide for instance a precise example showing that the considered extension is absolutely non-trivial.
 Recall from \cite{Gh13ITSF} that the univariate complex Hermite polynomials in \eqref{chp} satisfies the addition formula of Runge type,
\begin{align}\label{rungeCExple}
2^{\frac{m+n}{2}}H_{m,n}\left(\frac{p+q}{\sqrt{2}},\frac{\overline{p+q}}{\sqrt{2}}\right)=m!n!\sum_{j=0}^{m}
\sum_{k=0}^{n}\frac{H_{j,k}(p,\overline{p})H_{m-j, n-k}(q,\overline{q})}{j!k!(m-j)!(n-k)!}.
\end{align}
Accordingly, such identity remains valid for $H^Q_{m,n}$ restricted to given slice $L_I=\R+I\R$; $I^2=-1$, this readily follows since in this case
$H^Q_{m,n}(p+q,\overline{p+q}) = H_{m,n}(p+q,\overline{p+q})$ for every $p,q\in L_I $.
Although this is not true for arbitrary $p,q\in \Hq$ as can be showed by considering particular case of $n=0$. In fact,
the right hand-side in \eqref{rungeCExple} reduces further to
$\sum_{j=0}^{m} \binom{m}{j} p^j q^{m-j} $
which is completely different from
$ 2^{\frac{m+n}{2}} H^Q_{m,0}\left(\frac{p+q}{\sqrt{2}},\frac{\overline{p+q}}{\sqrt{2}}\right) = \left(p+q\right)^m$
for 
lack of commutativity in $\Hq$.

The paper is organized as follows. In Section 2, we review briefly some needed facts on the slice derivative.
Section 3 is devoted to derive an integral representation and give the explicit expression in terms of some special functions.
We also establish the connection with the real Hermite polynomials.
Rodrigues' representation and their alternatives involving the slice derivatives are given and then used
to discuss some auxiliary results, including three recurrence formulas, symmetry relations,  the realization as the exponential of
the Laplace-Beltrami operator acting on monomials $q^m\bq^n$ and operational formulas of Burchnall types.
Moreover, quadratic recurrence formulas of Nielsen type are presented.
In Section 4, we prove Theorem \ref{1} asserting that $H^Q_{m,n}(q,\bq )$ form an orthogonal basis of the Hilbert space $L^2(\Hq;e^{-|q|^2}d\lambda)$
of all square integrable $\Hq$-valued functions on $\Hq$ with respect to the Gaussian measure.
We also calculate the formal adjoint of the slice derivative in $L^2(\Hq;e^{-|q|^2}d\lambda)$ (see Proposition \ref{adjoint}).
Section 5 is devoted to establish some generating functions for the $H^Q_{m,n}(q,\bq )$.
We conclude the paper with an appendix in which we provide proofs of some elementary lemmas that we have used to prove our main results.

\section{The slice derivative} 

In this section, we review briefly some basic mathematical concepts relevant to the notion of the 
slice derivative \cite{GentiliStruppa07,ColomboSabadiniStruppa2011}.
Let $\Hq=\R^4$ denote the algebra of quaternions with basis elements $i_{0}=1$, $i_{1}$, $i_{2}$, $i_{3}=i_{1}i_{2}$ such that
$i_{\ell}^2=-1$; $\ell=1,2,3$ and $i_{1}i_{2} = - i_{2}i_{1}$. The conjugate of $q$ is the quaternion
 $\bq  = x_{0}-x_{1}i_{1}-x_{2}i_{2}-x_{3}i_{3}.$
Every generic element $q = x_{0}+x_{1}i_{1}+x_{2}i_{2}+x_{3}i_{3}$, with $x_{\ell}\in \R$, can be rewritten as
$q = x + yI,$ where $x, y$ are real numbers with $y \geq 0$ and
$I=I_q\in  \Sq =\{q \in \Hq; \, q^2=-1\}. $

  For any fixed unit quaternion $I \in \Sq$, the slice $L_I=\R+\R I$ is defined to be the complex plane in $\Hq$ passing through $0$, $1$ and $I$.
   Thus $\Hq$ can be seen as the union of all slices.
The left slice derivative $\partial_{s} f$ of a real differentiable function $f: \Omega \longrightarrow \Hq$
on a given domain $\Omega\subset \Hq$, is defined by
\begin{eqnarray}\label{cullend}
\partial_{s} f(q) =
\left\{
\begin{array}{ll}
\dfrac{1}{2}\left(\dfrac{\partial f_{I_q} }{\partial x}-I_q\dfrac{\partial f_{I_q} }{\partial y}\right)  (q),
                & \mbox{if} \, q \in  \Omega_{I_q}\setminus \R, \\
\dfrac{df}{dx}(x_q) ,
                &\mbox{if }  \, q=x_q \in \Omega_{I_q} \cap \R,
               \end{array}
            \right.
\end{eqnarray}
where $f_I$ denotes the restriction of $f$ to $\Omega_I := \Omega \cap L_I$.
Analogously, $\overline{\partial}_{s}$ is defined by
\begin{eqnarray}\label{cullendb}
\overline{\partial}_{s} f(q)  =
\left\{
\begin{array}{ll}
\dfrac{1}{2}\left(\dfrac{\partial f_{I_q} }{\partial x} + I_q\dfrac{\partial f_{I_q} }{\partial y}\right) (q),
                & \mbox{if}  \, q \in  \Omega_{I_q}\setminus \R , \\
\dfrac{df}{dx}(x_q) ,
                &\mbox{if }  \, q=x_q \in \Omega_{I_q} \cap \R.
               \end{array}
            \right.
\end{eqnarray}

In a similar way the right slice derivatives $\partial_{s}^R f$ and $\overline{\partial}_{s}^R$ are defined by
\begin{eqnarray}
\partial_{s}^R f(q)  & = \dfrac{1}{2}\left(\dfrac{\partial f_{I_q} }{\partial x}-\dfrac{\partial f_{I_q} }{\partial y} I_q \right)  (q),
\label{rcullend} \\
\overline{\partial}_{s}^R f(q) & = \dfrac{1}{2}\left(\dfrac{\partial f_{I_q} }{\partial x} + \dfrac{\partial f_{I_q} }{\partial y} I_q \right) (q), \label{rcullendb}
\end{eqnarray}
for $q=x_q+y_qI_q \in  \Omega\setminus \R$.

\begin{remark}\label{remFrontOf}
The $I_q$ in front of the derivative $\dfrac{\partial }{\partial y}$ depends on the point at which the function is being considered.
The notion of slice derivative is the key tool used by Gentili and Struppa to introduce the new theory of quaternionic slice regular functions
 (see \cite{GentiliStruppa07}). 
\end{remark}

\begin{remark}\label{remRLsliced}
The right $\partial_{s}^R f$ and the left $\partial_{s} f$ slice derivatives coincide when acting on expansion series with real coefficients,
like the Gaussian function $e^{-|q|^2}$.
\end{remark}


The following product rule \eqref{Product} for the slice derivative follows immediately from \eqref{Leibniz0} which can be obtained easily by direct computation starting from  \eqref{cullend}.
It will be used systematically in the next sections.

\begin{lemma}\label{KeyLemma}
We have
\begin{align}\label{Leibniz0}
\partial_{s} (fg)(q) &=\partial_{s} (f)(q) g(q)+f(q)\partial_{s} (g)(q) + \frac 12[f,I_q] \dfrac{\partial g_{I_q} }{\partial y} (q)
\end{align}
and therefore the following product rule
\begin{align}\label{Product}
\partial_{s} (fg)(q) =\partial_{s} (f)(q) g(q)+f(q)\partial_{s} (g)(q)
\end{align}
holds for every real differentiable functions $f,g:\Hq \longrightarrow \Hq$ such that $[f,I_q]:= f I_q - I_q f =0$, and in particular for  $f(q)= q^m\bq ^n e^{-|q|^2}$.
\end{lemma}

Subsequently, the analogue of the classical Leibniz formula for the slice derivative holds true under the assumptions of Lemma \ref{KeyLemma}. 

\begin{lemma}\label{lemLeibnizslice}
For every real differentiable functions $f,g:\Hq \longrightarrow \Hq$ such that $[f,I_q]:= f I_q - I_q f =0$, we have
\begin{align}\label{Leibnizslice}
\partial_{s}^n (fg)(q) = \sum_{j=0}^{n} \binom{n}{j}\partial_{s}^{n-j} (f)(q) \partial_{s}^{j} (g)(q).
\end{align}
\end{lemma}


Notice that $\partial_{s} (fg)\ne \partial_{s} (gf)$ in general, and that the condition  $[f,I_q] =0$ in Lemmas
 \ref{KeyLemma} and \ref{lemLeibnizslice} is satisfied when $f$ is a real-valued function and more generally when $f$ has convergent series
$$ f(q) = \sum_{m,n\geq 0} \alpha_{m,n} q^m\bq ^n,$$
with real coefficients, $\alpha_{m,n}\in \R$.
This is the case of the exponential function with quaternion variable defined as $e^q=\sum_{n=0}^{\infty}\frac{q^n}{n!}$.
In the sequel, we need to consider the following
$$e_{*}^{(p,q)}:=\sum_{n=0}^{\infty}\dfrac{p^{n}q^{n}}{n!}$$
for $p,q \in \Hq $. The series converges absolutely and uniformly on compact subsets of $\Hq$.
The function $e_{*}^{(p,\bq)}$ is the reproducing kernel of the slice hyperholomorphic Bargmann-Fock space \cite{AlpayColomboSabadiniSalomon2014,DikiGhanmi2016}.
Notice also that the function $e_{*}^{(p,q)}$ satisfies $\overline{e_{*}^{(p,q)}}=e_{*}^{(\bq,\overline{p})}$ and reduces further the usual exponential when $p\in \Hq$ and $q\in \R$,
$e_{*}^{(p,q)}=e_{*}^{(q,p)}=e^{pq}$. Moreover, we prove the following special function

\begin{lemma}\label{Lem:expstar}
For $\lambda\in \R$ and $u,v\in \Hq$, we have
\begin{align*}
e_{*}^{(u,\lambda+v)}= e_{*}^{(\lambda,u)} e_{*}^{(u,v)} = e^{\lambda u} e_{*}^{(u,v)}
\end{align*}
and 
\begin{align*}
e_{*}^{(\lambda+v,u)}=e_{*}^{(\lambda,u)} e_{*}^{(v,u)}= e^{\lambda u}  e_{*}^{(v,u)} .
\end{align*}
\end{lemma}

The above result is needed to prove of Theorem \ref{thm:GFxu}. Its proof is given in the appendix.

\section{Some auxiliary results for $H^Q_{m,n}(q,\bq )$}
The few known properties of $H^Q_{m,n}(q,\bq )$ are derived from their analogues of $H_{m,n}(z,\bz )$ by means of
        \begin{equation}\label{decompUq}
         H^Q_{m,n}(q,\bq )  = U_{q} \left(\begin{array}{cc} H_{m,n}(z,\bz ) & 0 \\ 0& \overline{H_{m,n}(z,\bz )} \\
         \end{array}\right)  U_{q}^{*},
         \end{equation}
where $q\in \Hq$ is identified here to its matrix representation
$q= U_{q} \left(\begin{array}{cc} z & 0 \\ 0&  \bz  \end{array}\right)  U_{q}^{*}$ for some $U_q \in SU(2)$ (see \cite{ThiruloAli13}). In the sequel, we discuss some basic properties of these polynomials using their different representations.

\subsection{Preliminary results.}
We provide here the explicit expression of the polynomials $H^Q_{m,n}(q,\bq )$ in terms of special functions as well as an integral representation.
We begin by presenting the connection of the slice poly-regular Hermite polynomials $H^Q_{m,n}(q,\bq)$ to the classical real Hermite polynomials
$$
H_{n}(x) =  (-1)^{n} e^{x^{2}} \dfrac{d^{n}}{dx^{n}}(e^{-x^{2}}).
$$

\begin{lemma}\label{realH}
For every $q=x+yI$  with $x,y \in \R$ and $I \in \Sq$, we have
\begin{align}\label{rhp}
H^Q_{m,n}(q,\bq )=\left(\dfrac{I}{2}\right)^{m+n}m!n!\sum_{k=0}^{m}\sum_{j=0}^{n}\dfrac{(-1)^{m+j}I^{k+j}}{k!j!(m-k)!(n-j)!}H_{k+j}(x)H_{m+n-k-j}(y).
  \end{align}
\end{lemma}

The following expresses $H^Q_{m,n}(q,\bq )$ in terms of the confluent hypergeometric function
${_1F_1}$ and the Laguerre polynomials $L_{k}^{(\alpha)}$.

\begin{lemma}\label{lag}
Set $m\wedge n:=min(m,n)$ and $m\vee n:=max(m,n)$. Then, we have
  \begin{align}
  H^Q_{m,n}(q,\bq)
  & =\dfrac{(-1)^{m\wedge n}\left(m\vee n\right)!}{\left(|m-n|\right)!} \frac{q^m\bq^n }{|q|^{2 m\wedge n}}
   {_1F_1}\left( \begin{array}{c} -m\wedge n\\ |m-n| +1 \end{array}\bigg | |q|^{2}\right) \nonumber \\ 
  &= (-1)^{m\wedge n}\left(m\wedge n\right)!  {|q|^{-2 m\wedge n}} q^m\bq^n L_{m\wedge n}^{(|m-n|)}(|q|^{2})
. \label{explicitLaguerre}
  \end{align}
\end{lemma}

\begin{remark}
By means of the global uniform estimate for the generalized Laguerre polynomials \cite{Maria}, 
$\left| L_{n}^{(\alpha)}(x)\right| \leq L_{n}^{(\alpha)}(0)e^{\frac{x}{2}}$ valid for $\alpha,x\geq 0$ and $n=0,1,\cdots$,
we obtain from \eqref{explicitLaguerre} the following upper bound
\begin{align}\label{estimate}
\left| H^Q_{n+p,n}(q,\bq) \right|\leq \dfrac{(n+p)!}{p!} |q|^{p} e^{\frac{|q|^2 }{2}}.
\end{align}
This estimate needed to justify different mathematical operations on series and integrals and in particular to ensure the
convergence of the series occurring in the next sections.
\end{remark}

We conclude this subsection by giving an integral representation of the polynomials

 \begin{lemma}\label{prop:IntRepr}
 For every $I\in \Sq$ and every $q\in L_I$, we have
 \begin{align}\label{IntRepr}
H^Q_{m,n}(q,\bq) = \frac{(\mp I)^{m+n}}{\pi} e^{|q|^2}  \int_{L_I} \xi^{m} \overline{\xi}^{n}  e^{-|\xi|^2 \pm I (\xi\bq + \overline{\xi} q) } , d\lambda_I(\xi)
\end{align}
where $ d\lambda_I$ is the Lebesgue measure on the slice $L_I$.
\end{lemma}

The proofs of the above lemmas are presented in the Appendix.

\subsection{Variants of Rodrigues' formula and recurrence formulas.}
We prove that the quaternionic Hermite polynomial $H^Q_{m,n}(q,\bq)$
defined by the Rodrigues' formula
         \begin{align} \label{RFQHP}
         H^Q_{m,n}(q,\bq )=(-1)^{m+n}e^{|q|^2 }\overline{\partial}_{s}^{m} \partial_{s}^{n} e^{-|q|^2}
         \end{align}
 on $\Hq\setminus \R$, and extended to the whole $\Hq$ by \eqref{explicitIntro},
 can be generated by iteration of the element $q^m $ using the first slice differential
operator $-\partial_{s}+\bq$. Namely, we assert

\begin{proposition}\label{GG}
The slice poly-regular Hermite polynomials $H^Q_{m,n}(q,\bq)$ have the representation
 \begin{align} \label{relc}
  H^Q_{m,n}(q,\bq) = \left(-\overline{\partial}_{s} + q\right)^{m} (\bq^{n})   =  \left(-\partial_{s} +\bq \right)^{n} (q^{m}).
 \end{align}
\end{proposition}

\noindent{\it Proof.}
The second equality in \eqref{relc} follows easily by setting $f(q)=q^{m}$ in the identity
   \begin{align}\label{z5}
      & \quad  \left(-\partial_{s} +\bq \right)^{n}(f)=
  (-1)^{n} e^{|q|^2 }\partial_{s}^n( e^{-|q|^2 }f),
     \end{align}
keeping in mind that $e^{-|q|^2 }$ is a real-valued function on $\Hq$ as well as
$$q^{m}e^{-|q|^2 }= (-1)^{m}\overline{\partial}_{s}^m (e^{-|q|^2 }).$$
The identity \eqref{z5} can be handled by induction starting from the right hand-side and making use of the rule product \eqref{Product}.
It should be noticed here that the left multiplication by $\bq $ and the left slice derivative $\partial_{s}$ 
 commute.
\fin\\

  \begin{remark}
Proposition \ref{GG} can be used to rederive the explicit expression of $H^Q_{m,n}(q,\bq )$ in terms of $q$ and $\bq$, or also in terms of the confluent hypergeometric function ${_1F_1}$. 
\end{remark}

\begin{lemma}
The operator $A_{q}=\overline{\partial}_{s}  $ is an annulation operator for the polynomials
$H^Q_{m,n}(q,\bq)$ for satisfying
  \begin{equation} \label{dbar}
\overline{\partial}_{s}  H^Q_{m,n}(q,\bq )=   n H^Q_{m,n-1}(q,\bq ).
\end{equation}
More generally, we have
  \begin{equation}
\label{amac}
     \partial^{j}_{s}\overline{\partial}^{k}_{s} H^Q_{m,n}(q,\bq )=
\displaystyle j!k! \binom{m}{j}\binom{n}{k} \left\{ \begin{array}{ll}
H^Q_{m-j,n-k}(q,\bq ) & \quad \mbox{if}\quad j \leq m, \, k \leq n;\\
0 & \quad \mbox{otherwise}.
\end{array}\right.
 \end{equation}
 \end{lemma}

\begin{proof}
The equality \eqref{dbar} is an  immediate consequence of equation \eqref{relc} combined with \eqref{z5}.
The symmetry relationship $ \overline{H^Q_{m,n}(q,\bq )} = H^Q_{n,m}(q,\bq)$ and \eqref{dbar} yields
   \begin{equation} \label{dd}
\partial_{s}  H^Q_{m,n}(q,\bq )=   m H^Q_{m-1,n}(q,\bq ).
\end{equation}
By induction we get \eqref{amac}.
\end{proof}

\begin{remark}
It is clear from \eqref{amac} that $\overline{\partial}_{s}^{m+1}  H^Q_{m,n}(q,\bq )=0$.
Thus, the polynomials  $H^Q_{m,n}$ are slice poly-regular functions of order $m$.
\end{remark}

Accordingly, $H^Q_{m,n}(q,\bq)$ are polynomials of degree $m$ in $q$ and degree $n$ in $\bq $ and satisfy the symmetry relationship
$H^Q_{m, n}(-q, \overline{-q})=(-1)^{m+n}H^Q_{m,n}(q,\bq ).$ Moreover, they verify the following three term recurrence formulas.

\begin{lemma}
We have
\begin{equation} \label{rf1}
H^Q_{m, n+1}(q,\bq )=-m H^Q_{m-1, n}(q,\bq )+\bq H^Q_{m, n}(q,\bq )
\end{equation}
and
\begin{equation}\label{rf2}
H^Q_{m+1,n}(q,\bq ) 
=-nH^Q_{m, n-1}(q,\bq )+ qH^Q_{m, n}(q,\bq ).
\end{equation}
Moreover, the polynomials  $H^Q_{m,n}(q,\bq)$ are common $L^2$-eigenfunctions of the second order
differential operators $\partial_{s}\overline{\partial}_{s} - \bq \overline{\partial}_{s}$
and $\partial_{s}\overline{\partial}_{s} - q \partial_{s}$.
\end{lemma}

\begin{proof}
The second three term recurrence formula \eqref{rf2} follows at once from the first, taking the conjugation.
Relation \eqref{rf1} follows by writing $H^Q_{m, n+1}$ as
   \begin{equation} \label{forrighr}
   H^Q_{m, n+1}(q,\bq ) 
 = \left(-\overline{\partial}_{s} + q \right)(H^Q_{m, n}(q,\bq )),
    \end{equation}
 thanks to \eqref{relc}, and using \eqref{dbar}. Next, by means of \eqref{rf2} combined with \eqref{amac} with $k=1$ and $j=0$ we get
    \begin{equation} \label{diffeq2}
  \left(-\partial_{s}\overline{\partial}_{s} + q \partial_{s} \right) H^Q_{m,n}(q,\bq)=  mH^Q_{m,n}(q,\bq).
\end{equation}
Similarly, we have
\begin{equation}\label{diffeq22}
\left(-\partial_{s}\overline{\partial}_{s} +  \bq  \overline{\partial}_{s} \right) H^Q_{m,n}(q,\bq)=  nH^Q_{m,n}(q,\bq).
  \end{equation}
\end{proof}

 \subsection{Exponential operational formula.}
We are concerned with the exponential representation for the slice poly-regular Hermite polynomials $H^Q_{m,n}(q,\bq )$. It is the quaternionic analogue of the operational representation for the univariate complex Hermite polynomials $H_{m,n}(z,\bz)$ obtained
  in \cite{IsmailTrans2016}.

\begin{proposition}\label{representation}
 We have the exponential representation
 $$H^Q_{m,n}(q,\bq )= e^{-\overline{\partial}_{s}\partial_{s}}\left(q^{m}\bq^{n}\right)=: e^{-\Delta_s} \left(q^{m}\bq^{n}\right).$$
\end{proposition}

\noindent{\it Proof.} Starting from \eqref{relc} and using the binomial formula for the commuting operators $-\overline{\partial}_{s} $ and $q$, we can rewrite $H^Q_{m,n}(q,\bq )$ as
\begin{align*}
H^Q_{m,n}(q,\bq ) =  \left(-\overline{\partial}_{s} + q\right)^{m} (\bq^{n})
= \sum_{k=0}^{m}   \binom{m}{k} (-1)^k \overline{\partial}_{s}^k (\bq^{n})   q^{m-k}.
\end{align*}
Now, since $\binom{m}{k} q^{m-k} = \partial_{s}^k(q^m)/k!$ when $k \leq n$ and vanishing otherwise, we get
$$
H^Q_{m,n}(q,\bq )
=\sum_{k=0}^{\infty}\frac{(-1)^{k}}{k!} \left(\partial_{s}\overline{\partial}_{s}\right)^{k} (q^m\bq^{n})
= e^{-\overline{\partial}_{s}\partial_{s}} \left(q^{m}\bq^{n}\right). \eqno{\mbox{\fin}}
$$

Formally, Proposition \ref{representation} when combined with \eqref{amac} shows that we can linearize $q^{m}\bq^{n}$ in terms of the polynomials
$ H^Q_{m-k,n-k}(q,\bq )$. More precisely, we assert

\begin{proposition} We have
\begin{align}\label{linearize}
q^{m}\bq^{n}&= m!n! \sum_{k=0}^{m\wedge n}  \frac{ H^Q_{m-k,n-k}(q,\bq ) }{k!(m-k)!(n-k)!} .
\end{align}
\end{proposition}

\noindent{\it Proof.} It suffices to show this for $m\geq n$. Direct computation using the explicit expression \eqref{explicitIntro} yields
\begin{align*}
m!n! \sum_{k=0}^{m}  \frac{ H^Q_{m-k,n-k}(q,\bq ) }{k!(m-k)!(n-k)!}
&= m!n! \sum_{k=0}^{m}\sum_{j=0}^{m-k} \frac{(-1)^{j}}{j!k!} \frac{  q^{m-k-j} \bq^{n-k-j} }{(m-k-j)!(n-k-j)!}
\\&= m!n! \sum_{s=0}^{m}\sum_{j=0}^{s} \frac{(-1)^{j}}{j!(s-j)!} \frac{ q^{m-s} \bq^{n-s} }{(m-s)!(n-s)!} .
\end{align*}
This immediately follows from $\sum_{k=0}^{m}\sum_{j=0}^{m-k}a_{j,k}=\sum_{s=0}^{m}\sum_{j=0}^{s}a_{j,s-j}$. Next, \eqref{linearize} follows by means of the fact $$\sum_{j=0}^{s}\frac{ (-1)^{j}}{j!(s-j)!}=  \delta_{s,0}. 
\eqno{\mbox{\fin}}$$

\subsection{Operational formulas of Burchnall type.}
Let us consider the differential operators
 \begin{align}
 \mathcal{A}_{m,n}(f)&=(-1)^{m}e^{|q|^2 } \overline{\partial}_{s}^{m} \left(\bq^{n}e^{-|q|^2 }f\right) \label{c}
  \end{align}
and
 \begin{align}
\mathcal{B}_{m,n}(f)&=(-1)^{m+n}e^{|q|^2 } \overline{\partial}_{s}^{m}\partial_{s}^{n}  \left(e^{-|q|^2 }f\right). \label{b}
 \end{align}

\begin{theorem}\label{s}
For given positive integers $m$ and $n$, we have the following operational formulas of Burchnall type involving $H^Q_{m,n}(q,\bq )$ for the sufficiently differentiable function $f$,
 \begin{align}
 \mathcal{A}_{m,n}(f)&=m!n!\sum_{j=0}^{m}\dfrac{(-1)^{j}}{j!}\frac{H^Q_{m-j,n}(q,\bq )}{(m-j)!n!} \overline{\partial}_{s}^{j}(f)
 \label{e}
 \end{align}
and
 \begin{align}
 \mathcal{B}_{m,n}(f)&=m!n!\sum_{j=0}^{m}\sum_{k=0}^{n}\dfrac{(-1)^{j+k}}{j!k!}\frac{H^Q_{m-j,n-k}(q,\bq )}{(m-j)!(n-k)!}
 \overline{\partial}_{s}^{j}\partial_{s}^{k} (f). \label{d}
 \end{align}
\end{theorem}

\noindent{\it Proof.}
To prove \eqref{e}, we apply the Leibniz formula for the slice derivative (Lemma \ref{lemLeibnizslice}) to \eqref{c} since $[\bq^{n}e^{-|q|^2 }, I_q]=0$ and next use \eqref{amac}. Thus, we obtain
 \begin{align*}
 \mathcal{A}_{m,n}(f)&=(-1)^{m}e^{|q|^2 }\sum_{j=0}^{m}\dfrac{m!}{j!(m-j)!}
\overline{\partial}_{s}^{m-j} \left(\bq^{n}e^{-|q|^2 }\right) \overline{\partial}_{s}^{j} (f)\\
 &=m!\sum_{j=0}^{m}\dfrac{(-1)^{j}}{j!(m-j)!}H^Q_{m-j,n}(q,\bq) \overline{\partial}_{s}^{j}(f).
 \end{align*}
The proof of \eqref{d} is quite similar. In fact, since $e^{-|q|^2 }$ is a real-valued function, we make use of the Leibniz formula to get
\begin{align*}
 \mathfrak{B}_{m,n}(f)&=(-1)^{m+n}e^{|q|^2 }\sum_{j=0}^{m}\sum_{k=0}^{n}\dfrac{m!n!}{j!k!(m-j)!(n-k)!}
\overline{\partial}_{s}^{m-j}\partial_{s}^{n-k} \left(e^{-|q|^2 }\right) \overline{\partial}_{s}^{j}\partial_{s}^{k}(f)\\
 &=m!n!\sum_{j=0}^{m}\sum_{k=0}^{n}\dfrac{(-1)^{j+k}}{j!k!(m-j)!(n-k)!}H^Q_{m-j,n-k}(q,\bq )
 \overline{\partial}_{s}^{j}\partial_{s}^{k}(f).
 \end{align*}
This yields the required result. 
 \fin

   \begin{remark}By \eqref{z5} and \eqref{RFQHP},  we can rewrite the differential operator $ \mathcal{B}_{m,n}(f)$ as
   $$  \mathcal{B}_{m,n}(f) =  \left(-\overline{\partial}_{s} + q\right)^{m}\left(-\partial_{s} +\bq \right)^{n}.(f).$$
  \end{remark}

   \begin{remark}
   The specific operators $\mathcal{A}_{m,0}$ and $\mathcal{B}_{m,0}$ are the same.
They are of particular interest since they lead to special examples of slice poly-regular functions. Indeed, by assuming that $f$ is a
slice regular function, we get from Theorem \ref{s} the following
 \begin{align*}
 (-1)^{n}e^{|q|^2 } {\partial}_{s}^{n} \left(e^{-|q|^2 }f\right) &= (-1)^{n}\sum_{k=0}^{n}(-1)^{k}\binom{ n}{k} \bq^k  {\partial}_{s}^{n-k}(f)
 = \sum_{k=0}^{n}  \bq^k h_{k}(q) ,
 \end{align*}
 where $h_{k}$ are slice regular functions.
  \end{remark}


\subsection{Special identities of Nielsen type.} \label{secId}
Quadratic recurrence identity of Nielsen type and their variants are proved using Burchnall's operational formulas.
We begin with the following

\begin{proposition}\label{W}
We have the following identity of Nielsen type
\begin{align*}
H^Q_{m+m^{'},n+n^{'}}(q,\bq )=m!n!m^{'}!n^{'}!\sum_{j=0}^{m\wedge m^{'}}
\sum_{k=0}^{n\wedge n^{'}}\dfrac{(-1)^{j+k}}{j!k!}\dfrac{H^Q_{m^{'}-j,n^{'}-k}(q,\bq )}{(m^{'}-j)!(n^{'}-k)!}
\dfrac{H^Q_{m-j,n-k}(q,\bq)}{(m-j)!(n-k)!}.
\end{align*}
\end{proposition}

\noindent{\it Proof.}
According to the Rodrigues' formula, we have
\begin{align*}
H^Q_{m+m^{'},n+n^{'}}(q,\bq )&=(-1)^{m+n+m^{'}+n^{'}}e^{|q|^2 }\partial^{m+m^{'}}_{s}\overline{\partial^{n+n^{'}}_{s}}\left(e^{-|q|^2}\right)\\
&=(-1)^{m^{'}+n^{'}}e^{|q|^2 }\partial^{m^{'}}_{s}\overline{\partial^{n^{'}}_{s}}\left(e^{-|q|^2}H^Q_{m,n}(q,\bq)\right).
\end{align*}
Then, making use of the operational formula \eqref{d}, we get
\begin{align*}
H^Q_{m+m^{'},n+n^{'}}(q,\bq )\stackrel{\eqref{d}}{=}m^{'}!n^{'}!\sum_{j=0}^{m^{'}}\sum_{k=0}^{n^{'}}\dfrac{(-1)^{j+k}H^Q_{m^{'}-j,n^{'}-k}(q,\bq )
 \partial^{k}_{s}\overline{\partial^{j}}_{s}(H^Q_{m,n}(q,\bq))}{j!k!(m^{'}-j)!(n^{'}-k)!}.
\end{align*}
Therefore, the result of Theorem \ref{W} follows since
$$ 
\partial^{k}_{s}\overline{\partial^{j}}_{s}(H^Q_{m,n}(q,\bq))=\dfrac{m!n!}{(m-k)!(n-j)!}H^Q_{m-k,n-j}(q,\bq).
\eqno{\mbox{\fin}}
$$

\begin{proposition} We have
\begin{align}
&\sqrt{2}^{m-n-n'}H^Q_{m, n+n'}(\sqrt{2}q,\sqrt{2}\bq) =\sum_{j=0}^{m} \binom{m}{j} H^Q_{m-j, n}(q,\bq)H^Q_{j,n'}(q,\bq),
\label{opcor1}
\\ &\sqrt{2}^{m+n-n'}H^Q_{m, n+n'}(\sqrt{2}q,\sqrt{2}\bq) =\sum_{j=0}^{m}\sum_{k=0}^{n} \binom{m}{j} \binom{n}{k} H^Q_{m-j, n-k}(q,\bq)H^Q_{j,k+n'}(q,\bq),
\label{opcor2}
\\&
H^Q_{m, n+n'}(q,\bq) = \sum_{j=0}^{m\wedge n'}  (-1)^{j}\frac{n'!}{(n'-j)!}\binom{m}{j} \bq^{{n'}-j}H^Q_{m-j, n}(q,\bq).
\label{opcor3}
\end{align}
\end{proposition}

\begin{proof}
The identity \eqref{opcor1} follows from \eqref{e} since for the particular case of $f=\bq^{n'}e^{-|q|^2}$ we have
$$\mathcal{A}_{m,n}(\bq^{n'}e^{-|q|^2})= \sqrt{2}^{m-n-{n'}} e^{-|q|^2}H^Q_{m, n+V{n'}}(\sqrt{2}q,\sqrt{2}\bq)$$
and $$\overline{\partial}_{s}^{j}(\bq^{n'}e^{-|q|^2})=(-1)^j e^{-|q|^2}H^Q_{j,{n'}}(q,\bq).$$
 While \eqref{opcor2} follows from \eqref{d} using the operator $\mathcal{B}_{m,n}$. Indeed, for $f=\bq^{n'}e^{-|q|^2}$, we have $$\mathcal{B}_{m,n}(\bq^{n'}e^{-|q|^2})= \sqrt{2}^{m+n-{n'}} e^{-|q|^2}H^Q_{m, n+{n'}}(\sqrt{2}q,\sqrt{2}\bq)$$
  and $$\overline{\partial}_{s}^{j}\partial_{s}^{k} (\bq^{n'}e^{-|q|^2})=e^{-|q|^2}H^Q_{j,k+{n'}}(q,\bq).$$
The last identity
\eqref{opcor3} is obtained since for $f=\bq^{n'}$, we have $$\mathcal{B}_{m,n}(\bq^{n'})= H^Q_{m, n+{n'}}(q,\bq)$$ and $$\overline{\partial}_{s}^{j}\partial_{s}^{k} (\bq^{n'})= \frac{n'!}{(n'-j)!}\bq^{n'-j}\delta_{k,0}$$
 for $j\leq n'$ and vanishes otherwise.
\end{proof}

\section{Orthogonality of $H^Q_{m, n}(q,\bq)$}
Added to the Lebesgue measure on $\Hq$, $d\lambda(q)=dx_0dx_1dx_2dx_3$; $q=x_0+x_1i_2+x_2i_3+x_3i_3$, we denote by $d\lambda_I(q)=dxdy$, $q=x+Iy$, the Lebesgue measure on a given slice $L_I$.
By $L^2(L_I;e^{-|q|^{2}}d\lambda_I)$ we denote the slice Hilbert space consisting of all $\Hq$-valued functions on $\Hq$ subject to norm boundedness
$\norm{f}_{slice}<+\infty$. This norm is induced from the inner product on a given slice
\begin{align} \label{slicedps}
\scal{f,g}_{slice} = \int_{L_I} f(x+Iy)\overline{g(x+Iy)} e^{-x^2-y^2}dxdy.
\end{align}
We also consider the full left quaternionic Hilbert space $L^{2}(\Hq;e^{-|q|^{2}}d\lambda)$ of square integrable functions on $\Hq$ with respect to the 
inner product
\begin{align} \label{ps}
\scal{f,g}_{full} :=\int_{\Hq } f(q)\overline{g(q)} e^{-|q|^2 }d\lambda(q).
\end{align}
The orthogonality of $H^Q_{m,n}(q,\bq)$ in the slice Hilbert space $L^2(L_I;e^{-|q|^{2}}d\lambda_I)$ is given by
\begin{align*}
\scal{H^Q_{m,n},H^Q_{j,k}}_{slice}  = m!n! \pi  \delta_{m,j} \delta_{n,k}
\end{align*}
and immediately follows from the one for the complex Hermite polynomials,
since the restriction of the $H^Q_{m,n}(q,\bq)$ to the slice $L_I$ reduces further to the complex Hermite polynomials. Moreover, the polynomials form an orthogonal basis of $L^2(L_I;e^{-|q|^{2}}d\lambda_I)$  (see \cite{Ito52,IntInt06,Gh08JMAA,ABEG2015}).
In the sequel, we provide orthogonality formulas in $L^{2}(\Hq;e^{-|q|^{2}}d\lambda)$ for the slice poly-regular  Hermite polynomials $H^Q_{m, n}(q,\bq)$ and prove their completeness in such Hilbert space.

\begin{proposition}\label{ortho}
The slice poly-regular Hermite polynomials $H^Q_{m,n}(q,\bq)$  form an orthogonal system
in 
the full Hilbert space $L^{2}(\Hq;e^{-|q|^{2}}d\lambda)$. More precisely, we have
\begin{align*}
\scal{H^Q_{m,n},H^Q_{j,k}}_{full}  = m! n! \pi Vol(\Sq)  \delta_{m,j} \delta_{n,k}  .
\end{align*}
\end{proposition}

\begin{proof}
To prove the orthogonal property of the polynomials $H^Q_{m, n}(q,\bq)$ in the full Hilbert space $L^{2}(\Hq;e^{-|q|^{2}}d\lambda)$,
we begin by rewriting \eqref{explicitLaguerre} in the polar coordinates $q=re^{I\Phi}$ with $r\leq 0$, $\Phi\in [0,2\pi)$ and $I\in \Sq$. Indeed, we have
$$H^Q_{m,n}(q,\bq) = (-1)^{m\wedge n}\left(m\wedge n\right)!   r^{|m-n|}e^{(m-n)I\Phi} L_{m\wedge n}^{(|m-n|)}(r^{2}),$$
 so that
\begin{align*}
\scal{H^Q_{m,n},H^Q_{j,k}}_{full} &:=\int_{\Hq } H^Q_{m,n}(q,\bq)\overline{H^Q_{j,k}(q,\bq)} e^{-|q|^2 }d\lambda(q)\\
&=c_{m,n}^{j,k} \int_{0}^\infty \int_{\Sq} \int_{0}^{2\pi} r^{|s|+|s'|}e^{[s-s']I\Phi} L_{m\wedge n}^{(|s|)}(r^{2})  L_{j\wedge k}^{(|s'|)}(r^{2}) e^{-r^2 }rd\Phi d\sigma(I)dr,
\end{align*}
where we have set $s=m-n$ and $s'=j-k$ and $c_{m,n}^{j,k}=(-1)^{m\wedge n+j\wedge k}\left(m\wedge n\right)! \left(j\wedge k\right)!$. Above $dr$ (resp. $d\Phi$) denotes the Lebesgue measure on positive real line (the unit circle) and $d\sigma(I)$ stands for the standard area element on $\Sq$.
Thus the polynomials $H^Q_{m, n}(q,\bq)$ and $H^Q_{j,k}(q,\bq)$ are orthogonal in $L^{2}(\Hq;e^{-|q|^{2}}d\lambda)$ whenever $s=m-n\ne j-k=s'$, and in particular for any $(m,n)\ne(j,k)$ such that $m-n\ne j-k'$. This readily follows since
$$\int_{0}^{2\pi}e^{[s-s']I\Phi} d\Phi = 2\pi \delta_{s,s'}.$$
 Now, for the pairs $(m,n),(j,k)$ such that $m-n = j-k$, we have $n=m+s$ and $k=j+s$ for some integer $s$. Thus, $m\wedge n=m $ (resp. $m\wedge n=n $)  if and only if $j\wedge k=j $ (resp. $j\wedge k=k$).
Therefore, if we assume that $m\wedge n=m $, then $n=m+s$ and $k=j+s$ with $s=0,1,2,\cdots$, and we obtain
\begin{align*}
\scal{H^Q_{m,n},H^Q_{j,k}}_{full} & =\scal{H^Q_{m,m+s},H^Q_{j,j+s}} \\
&= 2\pi(-1)^{m+j}m!j!   \int_{0}^\infty \int_{\Sq}  r^{2k}  L_{m}^{(s)}(r^{2})L_{j}^{(s)}(r^{2}) e^{-r^2} r dr \\
   &=\pi Vol(\Sq) (-1)^{m+j}m!j!    \int_{0}^\infty   L_{m}^{(s)}(t)L_{j}^{(s)}(t) t^{s} e^{-t}  dt.
\end{align*}
By means of the orthogonality property of the generalized Laguerre polynomials \cite{Rainville71}
$$\int_{0}^\infty   L_{m}^{(\alpha)}(t)L_{j}^{(\alpha)}(t) t^{\alpha} e^{-t}  dt  = \frac{\Gamma(\alpha+m+1)}{m!}\delta_{m,j},$$
it follows that for $m-n= j-k$, have
\begin{align*}
\scal{H^Q_{m,n},H^Q_{j,k}}_{full}  
= m! n! \pi Vol(\Sq)  \delta_{m,j}\delta_{n,k}  .
\end{align*}
\end{proof}

Accordingly, we are in position to prove the following interesting theorem.

\begin{theorem}\label{1}
The slice poly-regular Hermite polynomials $H^Q_{m,n}(q,\bq)$ constitute a complete orthogonal system in $L^{2}(\Hq;e^{-|q|^{2}}d\lambda)$.
In particular, for every $f\in L^{2}(\Hq;e^{-|q|^{2}}d\lambda)$ we have a unique left decomposition
\begin{align*}
 f(q)= \sum_{m=0}^{\infty}\sum_{n=0}^{\infty} C_{m,n} H^Q_{m,n}(q,\bq),
\end{align*}
 for certain quaternionic sliced constants $C_{m,n}=C_{m,n}^I$ satisfying the growth condition
\begin{align*}
\sum_{m,n=0}^\infty m!n! \left(\int_{\Sq} |C_{m,n}|^2 \right) <+\infty.
\end{align*}
\end{theorem}

\begin{remark}
The constants $C_{m,n}$ involved in Theorem \ref{1} are sliced in the sense of Remark \ref{remFrontOf}, which means that they may depend on the slice of the point at which the function is being considered, $C_{m,n}=C_{m,n}(I_q)$.
 \end{remark}

In order to prove Theorem \ref{1}, we need the following lemma.
\begin{lemma} \label{Ga}
 For every $a\in \R$ and $q\in \Hq$ we have
 $$ G_{a}(q) : =\sum_{m,n=0}^{\infty}\dfrac{a^{m+n}}{m!n!} H^Q_{m,n}(q,\bq)=e^{2a\Re(q) - a^{2}}.$$
 \end{lemma}

\noindent{\it Proof.}
Lemma \ref{Ga} appears as a particular case of the generating function \eqref{a3} below. However, we provide below a direct proof based on  the integral representation \eqref{IntRepr}. Indeed, we have
\begin{align*}
G_{a}(q)
& = \frac{e^{|q|^2}}{\pi}  \int_{L_I} \left( \sum_{m,n=0}^{\infty}\dfrac{(\mp a I)^{m+n}\xi^{m} \overline{\xi}^{n}}{m!n!} \right)     e^{-|\xi|^2 \pm I (\xi\bq + \overline{\xi} q) } d\lambda_I(\xi)
\\ & = \frac{e^{|q|^2}}{\pi}  \int_{L_I} e^{-|\xi|^2  \mp  I(a- \bq)\xi + \mp I (a - q)\overline{\xi}  } d\lambda_I(\xi)
\\ & = \frac{e^{|q|^2}}{\pi}  \pi  e^{ -(a- \bq)(a - q)  }  .
\end{align*}
This completes the proof of Lemma \ref{Ga}. The last equality follows thanks to
  \begin{align}\label{intGaussab}
  \int_{L_I} e^{-\nu|\xi|^2 + \alpha \xi + \beta \overline{\xi} } d\lambda_I(\xi) =  \left(\frac{\pi}{\nu}\right) e^{\frac{\alpha\beta}{\nu}},
  \end{align}
  valid for every fixed positive real number $\nu>0$ and arbitrary complex numbers $\alpha,\beta\in\C$.
   Formula \eqref{intGaussab} is quite easy to check by writing $\xi$ as $\xi =x+iy$; $x,y\in\R$, and next making use of the Fubini's theorem as well as the explicit formula for the Gaussian integral 
  \begin{align*}
  \int_{\R} e^{-\nu x^2 + bx} dx = \left(\frac{\pi}{\nu}\right)^{\frac{1}{2}}e^{\frac{b^2}{4\nu}}; \quad \nu>0, \, b\in\C.
  \end{align*}
\fin

\noindent{\it Proof of Theorem \ref{1}.}
We need only to prove completeness.
Let $f\in L^{2}(\Hq;e^{-|q|^{2}}d\lambda)$ and assume that for every $m,n$ we have
$\scal{ f,H^Q_{m,n}}_{full}=0$.
Thus, it follows that
$$
\int_{\Hq} f(q) e^{2a\Re(q) - a^{2} -|q|^{2}} d\lambda(q) 
= \scal{f,G_{a}}_{full}=\sum_{m,n=0}^{\infty}\dfrac{a^{m+n}}{m!n!} \scal{ f,H^Q_{m,n}}_{full}=0.
$$
By rewriting $q$ as $q=\sum_{\ell=0}^3 x_{\ell}i_{\ell}$, we obtain
$e^{2a\Re(q) - a^{2} -|q|^{2}} = e^{-(a-x_{0})^{2}-x_{1}^{2}-x_{2}^{2}-x_{3}^{2}} = e^{-\|X-T\|_{2}^{2}}=:\phi_{0}(X-T)$, where we have set $X=(x_{0},x_{1},x_{2},x_{3})$ and $T=(a,0,0,0)$. Therefore, every the component function $f_{\ell}$ of $f(q)=\sum_{\ell=0}^3 f_{\ell}(x_0,x_1,x_2,x_3) i_{\ell}=\sum_{\ell=0}^3 f_{\ell}(X) i_{\ell}$, satisfies
\begin{align}\label{convolution}
f_\ell \ast \phi_{0}(T)= \int_{\R^{4}}f_\ell(X) e^{-\|X-T\|_{2}^{2}}dX = 0,
\end{align}
Hence by applying the standard fourth-dimension Fourier transform $\mathfrak{F}$
to the both sides of \eqref{convolution}, we get $\mathfrak{F}(f_\ell \ast \phi_{0}) = \mathfrak{F}(f_\ell) \times \mathfrak{F}(\phi_{0})=0$.
Therefore, $\mathfrak{F}(f_\ell)$ is identically zero on $\R^{4}$, since $\mathfrak{F}(\phi_0)$ does not vanishes on $\R^4$ for being a Gaussian function.
The injectivity of the Fourier transform shows that $f_\ell$ is identically zero on $\Hq$ and so is $f$.
This proves that $\left\{H^Q_{m,n}\right\}_{m,n=0}^{\infty}$ is a complete system in $L^{2}(\Hq;e^{-|q|^{2}}d\lambda)$. To conclude for the proof, notice that by expanding $f,g\in L^{2}(\Hq;e^{-|q|^{2}}d\lambda)$ as $ f(q) = \sum_{m,n=0}^\infty a_{m,n} H^Q_{m,n}(q,\bq)$ and $g(q) = \sum_{m,n=0}^\infty b_{m,n} H^Q_{m,n}(q,\bq)$, and next using the orthogonality of the slice poly-regular Hermite polynomials in the slice Hilbert space $L^2(L_I,e^{-|q|^2}dxdy)$ (see Proposition \ref{ortho}), we obtain
\begin{align}\label{scalerfgfull}
\scal{f,g}_{full}
&= \sum_{m,n=0}^\infty \left(\int_{\Sq} a_{m,n} \overline{b_{m,n}} d\sigma(I)\right)  \norm{H^Q_{m,n}}_{L^{2}(L_I;e^{-|q|^{2}}dxdy)}^2 \nonumber \\
&= \pi \sum_{m,n=0}^\infty m!n! \left(\int_{\Sq} a_{m,n} \overline{b_{m,n}} d\sigma(I)\right) .
\end{align}
Thus, $ f(q) = \sum_{m,n=0}^\infty a_{m,n} H^Q_{m,n}(q,\bq)$ belongs to $L^{2}(\Hq;e^{-|q|^{2}}d\lambda)$ if and only if
$$ \sum_{m,n=0}^\infty m!n! \left(\int_{\Sq} |a_{m,n}|^2 \right) <+\infty.
\eqno{\mbox{\fin}}
$$

 We conclude this section with the following result giving the formal adjoint of $\partial_s$ in the Hilbert space
 $L^{2}(\Hq;e^{-|q|^{2}}d\lambda)$.

\begin{proposition}\label{adjoint}
For every $f,g\in L^{2}(\Hq;e^{-|q|^{2}}d\lambda)$, we have
$$\scal{ \overline{\partial_{s}}^R f, g }_{full}= \scal{ f,(-\partial_{s}^R+ M_{\bq}^R) g }_{full},$$
where $\partial_{s}^R$ and $\overline{\partial_{s}}^R$ are as in \eqref{rcullend} and  \eqref{rcullendb}, respectively, and $M_{\bq}^R$ is the right multiplication operator by $\bq$, $M_{\bq}^R g (q)= g(q) \bq $.
\end{proposition}

\noindent{\it Proof.}
Let $f,g\in L^{2}(\Hq;e^{-|q|^{2}}d\lambda)$ and expand them as
$$ f(q) = \sum_{m,n=0}^\infty a_{m,n} H^Q_{m,n}(q,\bq) \quad \mbox{and} \quad  g(q) = \sum_{j,k=0}^\infty b_{j,k} H^Q_{j,k}(q,\bq)  ,$$
 according to Theorem \ref{1}.
Subsequently,  making use of $(-\partial_{s}^R+ M_{\bq}^R)[b_{j,k} H^Q_{m,n}(q,\bq)]= b_{j,k}  H^Q_{m,n+1}(q,\bq)$, the right analogue of \eqref{forrighr}, we obtain
$$ (-\partial_{s}^R+ M_{\bq}^R) g(q) = \sum_{j,k=0}^\infty b_{j,k} H^Q_{j,k+1}(q,\bq)  .$$
Now, by means of \eqref{scalerfgfull}, we get
\begin{align}
\scal{ f,(-\partial_{s}^R+ M_{\bq}^R) g}_{full}
 = \sum_{m,k=0}^{\infty}  m! (k+1)! \left(\int_{\Sq} a_{m,k+1} \overline{b_{m,k}} d\sigma(I)\right).
 \label{eq1adj}
\end{align}
On the other hand, using \eqref{dbar}, $\overline{\partial_{s}} H^Q_{m,n} (q,\bq) = n H^Q_{m,n-1}(q,\bq) $,
as well as \eqref{scalerfgfull}, one gets
\begin{align}
\scal{ \overline{\partial_{s}}R f ,  g}_{full}
&= \sum_{m,n=0}^\infty m!n! \left(\int_{\Sq} (n+1) a_{m,n+1} \overline{b_{m,n+1}} d\sigma(I)\right)\nonumber
\\&= \sum_{m,n=0}^\infty m!k! \left(\int_{\Sq}  a_{m,k} \overline{b_{m,k}} d\sigma(I)\right). \label{eq2adj}
\end{align}
The proof is completed by comparing the right hand-sides in \eqref{eq1adj} and \eqref{eq2adj}.
\fin

\section{Generating functions} \label{s6}

The following lemma provides the action of the operator $e^{-\partial_{s}\overline{\partial}_{s}}$ on the exponential function $e^{-\lambda |q|^2}$ for $0<\lambda <1$. Whose the proof is easy to handel (see the appendix).

\begin{lemma}\label{Lem-gaussian}
Let $0  < \lambda < 1$. Then, we have:
\begin{align}\label{Formula-gaussian}
e^{-\Delta_s} \left(e^{-\lambda |q|^2}\right)=\frac{1}{1-\lambda}\exp\left(\dfrac{-\lambda |q|^2}{1-\lambda}\right).
\end{align}
\end{lemma}

\begin{proposition} We have the generating functions involving $H_{k,k}(q,\bq)$,
\begin{align}\label{genFct11}
 \sum_{k=0}^{\infty} \frac{(-1)^k\lambda^k}{k!} H^Q_{k,k}(q,\bq) = \frac{1}{1-\lambda} \exp\left(\dfrac{-\lambda |q|^2}{1-\lambda}\right)
\end{align}
and
\begin{align}\label{genFct12}
 \sum_{k=0}^{\infty} \frac{\lambda^k}{k!} H^Q_{k,k}(\sqrt{\lambda}q,\sqrt{\lambda}\bq)=   \frac{e^{\lambda |q|^2}}{1-\lambda} \exp\left(\dfrac{-\lambda |q|^2}{1-\lambda}\right).
\end{align}
\end{proposition}

\begin{proof}
Starting from the left hand-side of \eqref{Formula-gaussian}, expanding $e^{-\lambda |q|^2}$ as series and using the exponential representation $e^{-\Delta_s} \left(q^{m}\bq^{n}\right)=H^Q_{m,n}(q,\bq )$, we get
$$ e^{-\Delta_s} \left(e^{-\lambda |q|^2}\right) = \sum_{k=0}^{\infty} \frac{(-\lambda)^k}{k!}e^{-\Delta_s}(|q|^{2k})
= \sum_{k=0}^{\infty} \frac{(-1)^k\lambda^k}{k!} H^Q_{k,k}(q,\bq).$$
This proves \eqref{genFct11}. While \eqref{genFct12} can be handled by  expanding the operator $e^{-\Delta_s} $. Indeed, we get
\begin{align*}
 e^{-\Delta_s} \left(e^{-\lambda |q|^2}\right) &=
\sum_{k=0}^{\infty} \frac{(-1)^k\overline{\partial}_{s}^k\partial_{s}^k}{k!} \left(e^{-\lambda |q|^2}\right)
= e^{-\lambda |q|^2}\sum_{k=0}^{\infty} \frac{\lambda^k}{k!} H^Q_{k,k}(\sqrt{\lambda}q,\sqrt{\lambda}\bq).
\end{align*}
\end{proof}

\begin{lemma}\label{12}
Fix $q\in \Hq$ and let $u,v$ belong to the slice $L_{I_q}$. Then, we have the exponential generating function
\begin{equation}\label{a3}
  \sum_{m=0}^{\infty}\sum_{n=0}^{\infty}\dfrac{u^{m}}{m!}\dfrac{v^{n}}{n!}H^Q_{m,n}(q,\bq )=e^{uq-uv+v\bq}.
\end{equation}
\end{lemma}

Under the assumption that $u,v$ belong to the slice $L_{I_q}$, the above formula reduces further to the generating function of the univariate complexe Hermite polynomials. A proof of it is given in the Appendix.

Using Lemma \ref{12}, one obtains the following generating functions of exponential type.

\begin{theorem}\label{thm:GFxu} For $x\in \R$ and  $u,q\in\Hq$, we have
\begin{equation}\label{b3}
\sum_{m=0}^{\infty}\sum_{n=0}^{\infty}\dfrac{x^{m}}{m!}H^Q_{m,n}(q,\bq)\dfrac{u^{n}}{n!}
=e^{xq}e_{*}^{(\bq, u)} e^{-xu}
\end{equation}
and
\begin{equation}\label{genvu}
\sum_{m=0}^{\infty}\sum_{n=0}^{\infty}\dfrac{\overline{u}^{m}}{m!}
H^Q_{m,n}(q,\bq)\dfrac{u^{n}}{n!}= e^{-|u|^2}\left|e_{*}^{(\bq,u)}\right|^2.
\end{equation}
\end{theorem}

\noindent{\it Proof.}
Write $u\in \Hq$ as $u=a+bJ$ with $a,b \in \R$ and $J\in \Sq$. Then, the left hand-side in \eqref{b3} becomes
$$\sum_{m=0}^{\infty}\sum_{n=0}^{\infty}\dfrac{x^{m}}{m!}H^Q_{m,n}(q,\bq)\dfrac{u^{n}}{n!}
=\sum_{m=0}^{\infty}\sum_{n=0}^{\infty}\sum_{i=0}^{\infty}\dfrac{x^{m}}{m!}\dfrac{a^{n}}{n!}H^Q_{m,n+i}(q,\bq)\dfrac{(bJ)^{i}}{i!},$$
thanks to the fact
$$\sum_{n=0}^{\infty}\sum_{j=0}^{n}A_{n-j,j}=\sum_{n=0}^{\infty}\sum_{j=0}^{\infty}A_{n,j}.$$
Now, using $H^Q_{m,n+i}(q,\bq)=\left(-\partial_{s}+\bq\right)^{i}\left(H^Q_{m,n}(q,\bq)\right)$ as well as
 Lemma \ref{12}, we obtain
\begin{align*}
\sum_{m=0}^{\infty}\sum_{n=0}^{\infty}\dfrac{x^{m}}{m!}H^Q_{m,n}(q,\bq)\dfrac{u^{n}}{n!}
&=\sum_{i=0}^{\infty}\left(-\partial_{s}+\bq\right)^{i}\left(\sum_{m=0}^{\infty}
\sum_{n=0}^{\infty}\dfrac{x^{m}}{m!}\dfrac{a^{n}}{n!} H^Q_{m,n}(q,\bq)\right)\dfrac{(bJ)^{i}}{i!}
\\& =\sum_{i=0}^{\infty}\left(-\partial_{s}+\bq\right)^{i}e^{xq-ax+a\bq}\dfrac{(bJ)^{i}}{i!}.
\end{align*}
Now, since
\begin{align*}
\left(-\partial_{s}+\bq\right)^{i}e^{xq-ax+a\bq}=\sum_{j=0}^{i}\dfrac{(-1)^{j}i!\bq^{i-j}x^{j}}{j!(i-j)!}(e^{xq-ax+a\bq}),
\end{align*}
it follows
\begin{align*}
\sum_{m=0}^{\infty}\sum_{n=0}^{\infty}\dfrac{x^{m}}{m!}H^Q_{m,n}(q,\bq)\dfrac{u^{n}}{n!}
&=e^{xq-ax+a\bq}\sum_{i=0}^{\infty}\sum_{j=0}^{\infty}\dfrac{\bq^{i}(bJ)^{i}}{i!}\dfrac{ (-xbJ)^{j}}{j!}.
\\&=e^{xq-ax+a\bq}e_{*}^{(\bq,bJ)} e^{-xbJ}
\\&=e^{xq}e_{*}^{(\bq,u)}e^{-xu}
\end{align*}
in view of Lemma \ref{Lem:expstar}. 
In order to prove \eqref{genvu}, let $v=x+yI$ with $x,y \in \R$ and $I \in \Sq $.
By proceeding in a similar way as for \eqref{b3}, we can rewrite the right hand-side of
\eqref{genvu} as a single sum. Indeed, by means of the generating function \eqref{b3}, we have
\begin{align*}
\sum_{m=0}^{\infty} \sum_{n=0}^{\infty} \dfrac{v^{m}}{m!}H^Q_{m,n}(q,\bq) \dfrac{u^{n}}{n!}
&=\sum_{m=0}^{\infty} \sum_{n=0}^{\infty}\sum_{i=0}^{\infty}\dfrac{x^{m}}{m!} \dfrac{(yI)^{i}}{i!} H^Q_{m+i,n}(q,\bq)\dfrac{u^{n}}{n!}\\
&=\sum_{i=0}^{\infty} \dfrac{(yI)^{i}}{i!} \left(-\overline{\partial}_{s}+q\right)^{i} \left( \sum_{m=0}^{\infty} \sum_{n=0}^{\infty}\dfrac{x^{m}}{m!}
 H^Q_{m,n}(q,\bq)\dfrac{u^{n}}{n!}\right)
 \\&=\sum_{i=0}^{\infty} \dfrac{(yI)^{i}}{i!} \left(-\overline{\partial}_{s}+q\right)^{i}  \left( e^{xq}e_{*}^{(\bq,u)}e^{-xu} \right).
 \end{align*}
Direct computation shows that
 \begin{align*}
\left(-\overline{\partial}_{s}+q\right)^{i}  \left( e^{xq}e_{*}^{(\bq,u)}e^{-xu} \right)
&=\sum_{j=0}^{i} \dfrac{i!(-1)^{j}}{j!(i-j)!} q^{i-j} e^{xq} e_{*}^{(\bq,u)} u^{j}e^{-xu}
\\&=\sum_{j=0}^{\infty} (-1)^{j} \dfrac{(yI)^{j}}{j!}q^{i}e^{xq} e_{*}^{(\bq,u)}u^{j}e^{-xu}.
\end{align*}
Therefore, we get
 \begin{align*}
\sum_{m=0}^{\infty} \sum_{n=0}^{\infty} \dfrac{v^{m}}{m!}H^Q_{m,n}(q,\bq)\dfrac{u^{n}}{n!}
&=\sum_{j=0}^{\infty}  \dfrac{(-yI)^{j}}{j!} \left(\sum_{i=0}^{\infty} \dfrac{(yI)^{i}q^{i}}{i!} \right)e^{xq}
e_{*}^{(\bq,u)}u^{j}e^{-xu}\\
&=\sum_{j=0}^{\infty}  \dfrac{(-yI)^{j}}{j!} e_{*}^{(yI,q)} e^{xq} e_{*}^{(\bq,u)}u^{j}e^{-xu}.
\end{align*}
Now, Lemma \ref{Lem:expstar} infers $e_{*}^{(yI,q)} e^{xq} = e_{*}^{(v,q)}$. Therefore, for the specific particular case of
$v=\overline{u}=a-Ib$ (i.e., with $x=a$, $y=-b$ and $I=J$),
we see that
$e_{*}^{[\overline{u},q]} e_{*}^{(\bq,u)} = \overline{e_{*}^{(\bq,u)}} e_{*}^{(\bq,u)} = |e_{*}^{(\bq,u)}|^2$
is real and then commutes with $(bJ)^{j}$.
This implies that
\begin{align*}
\sum_{m=0}^{\infty} \sum_{n=0}^{\infty} \dfrac{v^{m}}{m!}H^Q_{m,n}(q,\bq)\dfrac{u^{n}}{n!}
&= \left|e_{*}^{(\bq,u)}\right|^2\sum_{j=0}^{\infty}\dfrac{(bJ)^{j} u^{j} }{j!}  e^{-au}
 = \left|e_{*}^{(\bq,u)}\right|^2 e_{*}^{(bJ,u)} e^{-au} .
\end{align*}
The required result follows bymeans of $e_{*}^{(bJ,u)} e^{-au} = e_{*}^{(-\overline{u},u)} =
e^{-|u|^2}$.
\fin

Consider the generating functions
\begin{align*}
G^m(q,\bq|v):= \sum_{n=0}^{\infty}\dfrac{ H^Q_{m,n}(q,\bq)v^n}{n!}
\end{align*} 
and
\begin{align*}
G^n(u|q,\bq):=\sum_{m=0}^{\infty}\dfrac{u^m H^Q_{m,n}(q,\bq)}{m!}.
\end{align*}
Their closed explicit expressions are given by the following

\begin{theorem} \label{thm:genfctv}
For every $u,v,q\in \Hq$, we have
\begin{align}
G^m(q,\bq|v)&= (q-v)^m_{*_R} *_R e^{(\bq,v)}_{*}, \label{gen1}
\\G^n(u|q,\bq) &=e^{(u,q)}_{*} *_L  (\bq-u)^m_{*_L}. \label{gen2}
\end{align}
\end{theorem}

For the proof, we will make use of the identity principle for slice regular functions
 \begin{lemma}[\cite{ColomboSabadiniStruppa2011,GentiliStoppatoStruppa2013}]\label{IdentityPrinciple}
Let $F$ be a slice regular function on a slice domain $\Omega$ and denote by $\mathcal{Z}_F$ its zero set.
If $\mathcal{Z}_F \cap \C_I$ has an accumulation point in $\Omega_I$ for some $I\in \Sq$, then $F$ vanishes identically on $\Omega$.
\end{lemma}

\noindent{\it Proof of Theorem \ref{thm:genfctv}.}
Notice first that $G^m(q,\bq|v)$ is right slice regular function in $v$ for every fixed $q\in \Hq$ and that $G(q,\bq|v) = (q-v)^m e^{v\bq}$, for every $v\in L_{I_q}$, by means of \cite[Proposition 3.4]{Gh13ITSF}.
The extension of $v \longmapsto  (q-v)^m $ (resp. $v \longmapsto e^{v\bq}$) to right slice regular function is given by
$v \longmapsto  (q-v)^m_{*_R} $ (resp. $v \longmapsto e^{(\bq,v)}_{*} $), where ${*_R} $ denotes the $*_R$-product
of right slice regular functions as defined in \cite{ColomboSabadiniStruppa2011}. Therefore, the function $v \longmapsto (q-v)^m_{*_R} *_R e^{(\bq,v)}_{*} $
is right slice regular and coincides with $(q-v)^m e^{v\bq}$ on the slice $L_{I_q}$. Then by the identity principle for right
slice regular functions (Lemma \ref{IdentityPrinciple}), we conclude that $G^m(q,\bq|v)= (q-v)^m_{*_R} *_R e^{(\bq,v)}_{*} $ for every $,v,q\in \Hq$.
The proof of \eqref{gen2}, for the left slice regular function $G^n(u|q,\bq)$ in $u$, can be handled in a similar way making use of the $*_L$-product for left slice regular functions.
\fin


We conclude this section by proving the following generating function involving both the real and slice poly-regular Hermite polynomials.
Namely, we have the following

\begin{theorem} \label{thm:bilgenfct}
We have
\begin{align*}
\sum_{n=0}^{\infty}\dfrac{H_{n}(x)H^Q_{m,n}(q,\bq)}{n!} =e^{-\bq^{2}+2x\bq} H^Q_{m}\left(\bq+\frac{q}{2}-x\right).
\end{align*}
\end{theorem}

\noindent{\it Proof.}
Making use of $H^Q_{m,n}(q,\bq)=e^{-\Delta_s} (q^{m}\bq^{n})$, we obtain
\begin{align}
\sum_{n=0}^{\infty}\dfrac{H_{n}(x)H^Q_{m,n}(q,\bq)}{ n! }
&= e^{-\Delta_s} \left( q^{m} \sum_{n=0}^{\infty}\dfrac{ \bq^{n} H_{n}(x) }{ n! } \right) \nonumber
\\ &= \sum_{j=0}^{m}\dfrac{(-1)^{j}}{j!} \frac{m!q^{m-j}}{(m-j)!}
\left(\sum_{n=j}^{\infty}\dfrac{\bq^{n-j}}{(n-j)!} H_{n}(x) \right) \nonumber
\\ &= \sum_{j=0}^{m}\dfrac{(-1)^{j}}{j!} \frac{m!q^{m-j}}{(m-j)!}
\left(\sum_{k=0}^{\infty}\dfrac{\bq^{k}}{k!} H_{k+j}(x) \right). \label{ama}
\end{align}
The last equality holds thanks to the change of indices $k=n-j$. Using the fact \cite[p.197]{Rainville71}
$$\sum_{k=0}^{\infty}\dfrac{\bq^{k}}{k!} H_{k+j}(x)= e^{-\bq^2 + 2x \bq } H_j(x-\bq),$$
we obtain
\begin{align}
\sum_{n=0}^{\infty}\dfrac{H_{n}(x)H^Q_{m,n}(q,\bq)}{ n! }
&= \sum_{j=0}^{m}\dfrac{(-1)^{j}}{j!} \frac{m!q^{m-j}}{(m-j)!}
\left(e^{-\bq^2 + 2x \bq } H_j(x-\bq) \right)
\\ &= e^{-\bq^2 + 2x \bq } \sum_{j=0}^{m} \binom{m}{j}    q^{m-j}  H_j(\bq-x) . \label{amab}
\end{align}
Finally, the desired result follows by utilizing the fact that
$$ \sum_{j=0}^{m} \binom{m}{j}   H_{j}(x) (2\xi)^{m-j} = H_{m}( x + \xi ). \eqno{\mbox{\fin}}$$

As immediate consequence we claim the following

\begin{corollary} \label{cor:bilgenfct}
We have
\begin{align*}
\sum_{m,n=0}^{\infty}\dfrac{t^m H_{n}(x)H^Q_{m,n}(q,\bq)}{m!n!}
=e^{-t^2 -\bq^{2} + 2(x+t)\bq  + tq  - 2tx}.
\end{align*}
\end{corollary}

\section{Appendix}

\subsection{Proof of Lemma \ref{lag}.}
 The second expression in Lemma \ref{lag} is an immediate consequence of the first one due to
the fact \cite[p. 200]{Rainville71}:
 $${_1F_1}\left( \begin{array}{c} -n \\  \alpha+1 \end{array} \bigg | x\right)=\dfrac{n!}{(\alpha+1)_{n}}L_{n}^{\alpha}(x).$$
To get the first one we start from
\begin{align} \label{explicit}
H^Q_{m,n}(q,\bq ) &= \sum_{k=0}^{m\wedge n} (-1)^k k! \binom{m}{k}  \binom{n}{k}   q^{m-k} \bq^{n-k},
\end{align}
 with the assumption that $m\geq n$, and we make the change of indices $n-k=i$. We obtain
\begin{align*}
H^Q_{m,n}(q,\bq )&=\sum_{i=0}^{n}\dfrac{(-1)^{i+n}n!m!}{(n-i)!(m-n+i)!}\dfrac{|q|^{2i} q^{m-n}}{i!}\\
&=\dfrac{(-1)^{n}m!}{(m-n)!}q^{m-n}\sum_{i=0}^{n}\dfrac{(-n)_{i}}{(m-n+1)_{i}}\dfrac{|q|^{2i}}{i!}\\
&=\dfrac{(-1)^{n}m!}{(m-n)!} q^{m-n}{_1F_1}\left( \begin{array}{c} -n\\  m-n+1 \end{array}\bigg ||q|^{2}\right).
\end{align*}
For the result when $n\geq m$, we use the previous one combined with $H^Q_{m,n} (q,\bq )= \overline{H^Q_{m,n}(q,\bq )}$ to get
$$
H^Q_{m,n}(q,\bq )=\bq ^{n-m}\dfrac{(-1)^{m}n!}{(n-m)!}{_1F_1}\left( \begin{array}{c} -m\\  n-m+1 \end{array}\bigg ||q|^{2}\right)
.\eqno{\mbox{\fin}}
$$

\subsection{Proof of Lemma \ref{realH}.}
This holds since $e^{|q|^2 }$ is real-valued. In fact, starting from \eqref{RFQHP} we obtain
\begin{align*}
H^Q_{m,n}(q,\bq) 
& = (-1)^{m+n} e^{x^2+y^2} \left( \frac 12 \left( \partial_{x} + I \partial_{y} \right)\right)^{m} \circ
 \left( \frac 12 \left( \partial_{x} - I \partial_{y} \right)\right)^{n} (e^{-x^2-y^2})
\\& = \left( -\frac 12\right)^{m+n} e^{x^2+y^2} \sum_{k=0}^n \binom{n}{k}  \left( \partial_{x} + I \partial_{y} \right)^{m} \circ
 \left(   \partial_{x}^{n-k}(e^{-x^2}) \partial_{y}^{k} (e^{-y^2})\right) ( - I )^{k}
\\& = \left( -\frac 12\right)^{m+n} \sum_{j=0}^m   \sum_{k=0}^n \binom{m}{j} \binom{n}{k}
e^{x^2} \partial_{x}^{m+n-j-k} (e^{-x^2})  e^{y^2} \partial_{y}^{j+k} (e^{-y^2}) I^j ( - I )^{k}
  \\& = \left( \frac 12\right)^{m+n} \sum_{j=0}^m   \sum_{k=0}^n (-1)^k \binom{m}{j} \binom{n}{k}  H_{m+n-j-k} (x)   H_{j+k} (y) I ^{j+k} .
  \end{align*}
\fin

\subsection{Proof of Lemma \ref{prop:IntRepr}.}
From Lemma \ref{realH}, the integral representation for the Hermite polynomials
 $$ H_{n}(x)=\dfrac{(\pm 2I)^{n}}{\sqrt{\pi}} \int_{\R} t^n e^{-(t \pm Ix)^{2}} dt,$$
 valid for every $I\in\Sq$,  
 and the binomial formula, we get
\begin{align*}
H^Q_{m,n}(q,\bq)
& = \left( \frac 12\right)^{m+n} \sum_{j=0}^m   \sum_{k=0}^n (-1)^k \binom{m}{j} \binom{n}{k}  H_{m+n-j-k} (x)   H_{j+k} (y) I ^{j+k}
\\
&  =  \dfrac{(\pm I)^{m+n} }{\pi} \int_{\R^2}  (t_1+It_2)^m  (t_1-It_2)^{n}  e^{-(t_1 \pm Ix)^{2}}   e^{-(t_2 \pm Iy)^{2}}   dt_1 dt_2
\\&  =  \dfrac{(\pm I)^{m+n} }{\pi} e^{  x^2 + y^2 } \int_{\R^2}  (t_1+It_2)^m  (t_1-It_2)^{n}  e^{-t_1^2 -t_2^2  \mp 2 I(xt_1+yt_2) }     dt_1 dt_2
  \end{align*}
  By setting $\xi = t_1+I t_2\in L_I$, we conclude that
  \begin{align*}
H^Q_{m,n}(q,\bq) &  =  \dfrac{(\pm I)^{m+n} }{\pi} e^{|q|^2} \int_{L_I}  \xi^m  \overline{\xi}^{n}  e^{-|\xi|^2  \mp 2 I \Re(\scal{\xi,q} )}     d\lambda(\xi)
  \end{align*}
This completes the proof.
Another direct proof of \eqref{IntRepr} can be given starting from \eqref{RFQHP}.
\fin

\subsection{Proof of Lemma \ref{Lem-gaussian}.}
Notice first that the operator $ \Delta_s =  \partial_{s}\overline{\partial}_{s}$ on $\Hq\setminus \R$ is closely connected to the Laplacian on $\R^2$. In fact, if $\Delta_x=\dfrac{1}{4} \dfrac{\partial^{2}}{\partial x^{2}}$ and $\Delta_y = \dfrac{1}{4}\dfrac{\partial^{2}}{\partial x^{2}}$, then $ \Delta_s = \Delta_x + \Delta_y $.Therefore, for $e^{-\lambda |q|^2}=e^{-\lambda x^{2}-\lambda y^{2}}$ with $q=x+yI$, we obtain
\begin{align*}
e^{-\Delta_s} \left(e^{-\lambda |q|^2}\right)&=\left(e^{-\Delta_x} ( e^{-\lambda x^{2}})\right)\times \left(e^{-\Delta_y} ( e^{-\lambda y^{2}})\right).
\end{align*}
The identity \eqref{Formula-gaussian} follows immediately making use of 
$$
e^{-\Delta_x} \left( e^{-\lambda x^{2}}\right)=\left(1-\lambda\right)^{-\frac{1}{2}}\exp\left(\dfrac{-\lambda x^{2}}{1-\lambda}\right).
 \eqno{\mbox{\fin}}
 $$

\subsection{Proof of Lemma \ref{12}.}
By means of the fact $H^Q_{m,n}(q,\bq )= e^{-\Delta_s} \left(q^{m}\bq^{n}\right)$, we get
\begin{align*}
\sum_{m=0}^{\infty}\sum_{n=0}^{\infty}\dfrac{u^{m}}{m!}\dfrac{v^{n}}{n!}H^Q_{m,n}(q,\bq )
=\sum_{m=0}^{\infty}\sum_{n=0}^{\infty}\dfrac{u^{m}}{m!}\dfrac{v^{n}}{n!} e^{-\Delta_s} \left(q^{m}\bq^{n}\right)
= e^{-\Delta_s}  \left(e^{uq}e^{v\bq}\right)
\end{align*}
for  $u,v$ belonging to the slice $L_{I_q}$.
Substitution of
\begin{align*}
 e^{-\Delta_s}  \left(e^{uq}e^{v\bq}\right)=\sum_{k=0}^{\infty}\frac{(-1)^{k}}{k!}e^{uq}u^{k}v^{k}e^{v\bq}
 =e^{uq}\left[\sum_{k=0}^{\infty}\frac{(-1)^{k}}{k!}(uv)^{k}\right]e^{v\bq}
\end{align*}
in the previous equation yields
$$\sum_{m=0}^{\infty}\sum_{n=0}^{\infty}\dfrac{u^{m}}{m!}\dfrac{v^{n}}{n!}H^Q_{m,n}(q,\bq )=e^{uq-uv+v\bq}.
\eqno{\mbox{\fin}}
$$

\noindent{\bf Acknowledgements:}
The assistance of the members of the Intissar's seminar on "Partial Differential Equations, Analysis and Spectral Geometry" is gratefully acknowledged.

\end{document}